\documentclass[10pt]{amsart}
\usepackage{amsmath}
\usepackage{paralist}
\usepackage{graphicx} 
\usepackage{epstopdf}
\usepackage[colorlinks=true]{hyperref}
\usepackage{float}
\usepackage{color,comment}
\usepackage{caption}
\usepackage{wrapfig}
\usepackage{multirow}
\usepackage{tabularx}
\usepackage{booktabs}
\usepackage{subfigure}
\newtheorem{theorem}{Theorem}[section]

\newtheorem{example}{Example}[section]

\theoremstyle{definition}
\newtheorem{definition}[theorem]{Definition}
\newtheorem{remark}{Remark}[section]


\begin{document}
\title[Square root functional response]{Comment on Predator-prey dynamical behavior and stability with Square root functional response}
\author[Bearden, Antwi-Fordjour]{}
\maketitle

\centerline{\scshape Kendall H. Bearden$^1$ and Kwadwo Antwi-Fordjour$^{*1}$}

\smallskip
\centerline{ *Corresponding author's email: kantwifo@samford.edu}

\medskip
{\footnotesize
  \centerline{1) Department of Mathematics and Computer Science,}
 \centerline{ Samford University,}
 \centerline{ Birmingham, AL 35229, USA.}
}

\begin{abstract}
In this research, we revisit the paper by Pal et al. [Int. J. Appl. Comput. Math (2017) 3:1833-1845] and comment on the claim that global stability of the interior equilibrium point depends on some key parameters. This is not true, and we have provided detailed proof to this effect. The considered model is a modified Lotka-Volterra model where square root functional response is involved. The square root functional response is non-differentiable at the origin and hence cannot be studied with standard local stability tools. Furthermore, our numerical simulations indicate that we can classify the phase portrait into two modes of behavior where some positive initial conditions converge towards the predator axis in finite time. 
\end{abstract}

\section{Introduction}
\noindent Finite time extinction of species is possible in systems with functional responses of the forms $f(x)=x^p$ and $g(x)=\frac{x^p}{c+x^p}$, where $0<p<1$ and $c$ is the half saturation constant. When $p=1$, $f(x)$ is the type I functional response and $g(x)$ the type II functional response. The square root functional response is a special case of $f(x)$ when $p=0.5$ and possess interesting complex dynamics. For further readings on the dynamical behavior of models exhibiting finite time extinction, please see \cite{BS19, KRM20,APTW23} and references therein.  \par

Furthermore, the square root functional response has gained tremendous importance in the past few years in particular after the publication by Braza \cite{B12}. According to Braza, the square root functional response originally proposed by Gauss \cite{G34} models systems where the prey species exhibit strong herd behavior during interaction with predators. Herd behavior refers to the phenomenon in which species in a group act together for a given period without coordination by a central authority \cite{APTW23}.
 

In \cite{P17}, authors considered the modified Lotka-Volterra model system where the prey species exhibit herd behavior. The model system is given by
\begin{equation}\label{EquationMain}
\begin{cases}
\dfrac{dx }{dt} &=rx(1-x)-  y\sqrt{x}, \qquad x(t)\geq 0,\\ \\
\dfrac{dy }{dt} &=-\alpha y + \beta y\sqrt{x} \qquad y(t)\geq 0.
\end{cases}
\end{equation}

\noindent The model system \eqref{EquationMain} was also considered by Braza \cite{B12} where $r=1$. The variables and parameters used in the model are defined in Table \ref{tab:table1}:
\begin{table}[h!]
  \begin{center}
    \caption{List of parameters used in the model. All parameters considered are positive constants.}
   \label{tab:table1}
    \begin{tabular}{@{}l l@{}}
     Parameter/Variables & Description \\
      \midrule
$x$                         & Prey population\\
$y$                         & Predator population\\
$r$                           & Birth rate of prey \\
$\alpha$                           & Intrinsic death rate of the predator population\\
$\beta$                        & Biomass conversion efficiency\\

     \bottomrule
    \end{tabular}
  \end{center}
\end{table}

\section{Stability Guidelines}
\noindent The model system \eqref{EquationMain} admits the following three non-negative equilibrium points.
\begin{itemize}
\item[(i)] Extinction equilibrium, $P_0=(0,0)$
\item[(ii)] Prey only equilibrium, $P_1=(1,0)$
\item[(iii)] Interior equilibrium, $P_2=(x^*,y^*)$, where $x^*=\left(\frac{\alpha}{\beta}\right)^2$ and $y^*=\frac{r\alpha (\beta^2-\alpha^2)}{\beta^3}$
\end{itemize}

We provide a recap of the local stability result of the interior equilibrium from \cite{P17}.

\begin{theorem}[Theorem 4, \cite{P17}]
Equilibrium point $P_2$ is locally asymptotically stable if and only if $\frac{\alpha}{\beta}>\frac{1}{\sqrt{3}}$ and unstable if $\frac{\alpha}{\beta}<\frac{1}{\sqrt{3}}$.
\end{theorem}

 The authors in \cite{P17} proved the global asymptotic stability of the interior equilibrium point $P_2(x^*,y^*)$ under certain parametric restriction. This is not true and the proof is not correct. 

\begin{theorem}[Theorem 7, \cite{P17}]
The interior equilibrium point $P_2(x^*,y^*)$ of the model system \eqref{EquationMain} is globally asymptotically stable if 
\begin{equation}\label{ineq:1}
\dfrac{\alpha}{\beta}>\dfrac{1}{\sqrt{3}}.
\end{equation}

\end{theorem}

Next, we shall show the existence of extinction of the prey species in finite time under the conditions provided for global stability of $P_2$ in \cite{P17}. Hence the interior equilibrium point is not globally asymptotically stable as claimed by the authors.

\subsection{Finite Time Extinction of Prey}

\begin{theorem}
The prey equation $x=x(t)$ with some carefully chosen initial conditions $x(0)>0,y(0)>0$ will go extinct in finite time if
\begin{equation*}
K(x(0))<y(0),
\end{equation*}
where $K(x(0))=(r+2\alpha)\sqrt{x(0)}$. Thus, the interior equilibrium point $P_2$ of the model system \eqref{EquationMain} cannot exist globally for any parameter restriction.
\end{theorem}

\begin{proof}
Consider the predator equation in the model system \eqref{EquationMain}
\begin{align*}
\dfrac{dy }{dt} &=-\alpha y + \beta y\sqrt{x}\\
& \geq -\alpha y
\end{align*}
since $ \beta y\sqrt{x}\geq 0$, it implies 
\begin{align}\label{y bound}
y(t)\geq y(0)e^{-\alpha t}
\end{align}

\noindent Next, the prey equation in the model system \eqref{EquationMain} gives
\begin{align*}
\dfrac{dx }{dt} &=rx(1-x)-  y\sqrt{x}\\
&\leq rx(1-x)
\end{align*}
which implies that 
\[x(t)\leq 1\]
for all $t>0$ via comparison to the logistic equation (i.e. $\frac{dx}{dt}=rx(1-x)$). Now
\begin{align*}
\dfrac{dx }{dt} &=rx(1-x)-  y\sqrt{x}\\
&\leq rx-\sqrt{x}y
\end{align*}

\noindent We divide both sides by $\sqrt{x}$ since $x$ is positive.
\begin{align}\label{ineq:e1}
\dfrac{dx }{dt}\dfrac{1}{\sqrt{x}} &\leq  r\sqrt{x}-y
\end{align}
Notice that the left hand side of \eqref{ineq:e1} can be rewritten as
\[\dfrac{dx }{dt}\dfrac{1}{\sqrt{x}} =\dfrac{1}{0.5}\dfrac{d(\sqrt{x})}{dt}\]
Thus
\begin{align*}
\dfrac{1}{0.5}\dfrac{d(\sqrt{x})}{dt} &\leq  r\sqrt{x}-y
\end{align*}

\noindent Applying the lower bound on $y$ from \eqref{y bound}, we obtain
\begin{align}
\dfrac{d(\sqrt{x})}{dt} &\leq  0.5 r\sqrt{x}-0.5 y(0)e^{-\alpha t} \\
\dfrac{d(\sqrt{x})}{dt} -  0.5 r\sqrt{x}&\leq -0.5 y(0)e^{-\alpha t} \label{ineq:a}
\end{align}

\noindent Clearly, from the inequality in \eqref{ineq:a},  $e^{-0.5rt}$ is the integrating factor. We multiply both sides of the inequality by the integrating factor.

\begin{align}
e^{-0.5rt}\left(\dfrac{d(\sqrt{x})}{dt} -  0.5 r\sqrt{x}\right) &\leq e^{-0.5rt}\left(-0.5 y(0)e^{-\alpha t}\right) \label{ineq:e}
\end{align}

\noindent Note that the left hand side can be simplified as
\begin{equation}\label{eqn:lhs}
e^{-0.5rt}\left(\dfrac{d(\sqrt{x})}{dt} -  0.5 r\sqrt{x}\right)=\dfrac{d\left(\sqrt{x}e^{-0.5rt}\right)}{dt}.
\end{equation}

\noindent Additionally, the right hand side can be simplified as

\begin{equation}\label{eqn:rhs}
e^{-0.5rt}\left(-0.5 y(0)e^{-\alpha t}\right)=-0.5y(0)e^{-(0.5r+\alpha)t}
\end{equation}

\noindent Substituting the results in \eqref{eqn:lhs} and \eqref{eqn:rhs} into the inequality in \eqref{ineq:e}, we obtain

\begin{align}\label{ineq:b}
\dfrac{d\left(\sqrt{x}e^{-0.5rt}\right)}{dt} &\leq -0.5y(0)e^{-(0.5r+\alpha)t}
\end{align}

\noindent Integrate both sides of \eqref{ineq:b} above in the time interval $[0,t]$.

\begin{align}\label{ineq:c}
\displaystyle\int_0^t \dfrac{d\left(\sqrt{x}e^{-0.5rs}\right)}{ds} ds &\leq -0.5y(0)\displaystyle\int_0^t e^{-(0.5r+\alpha)s} ds
\end{align}

\noindent By application of the Fundamental Theorem of Calculus, the left hand side of \eqref{ineq:c} becomes

\begin{align}\label{eqn:lhs2}
\displaystyle\int_0^t \dfrac{d\left(\sqrt{x}e^{-0.5rs}\right)}{ds} ds &= \sqrt{x(t)}e^{-0.5rt}-\sqrt{x(0)}e^{-0.5r(0)}  \nonumber \\
&= \sqrt{x(t)}e^{-0.5rt}-\sqrt{x(0)}
\end{align}

\noindent Also, the integral on the right hand side of \eqref{ineq:c} becomes

\begin{align}
-0.5y(0)\displaystyle\int_0^t e^{-(0.5r+\alpha)s} ds &=\dfrac{0.5y(0)}{0.5r+\alpha}\left( e^{-(0.5r+\alpha)t}-e^{-(0.5r+\alpha)(0)} \right) \nonumber \\
&=- \dfrac{0.5y(0)}{0.5r+\alpha}\left( 1-e^{-(0.5r+\alpha)t}\right) \label{eqn:5}
\end{align}

\noindent Thus, the inequality in \eqref{ineq:c} becomes
\begin{align}\label{ineq:w}
\sqrt{x(t)}e^{-0.5rt} &\leq \sqrt{x(0)}  - \dfrac{0.5y(0)}{0.5r+\alpha}\left( 1-e^{-(0.5r+\alpha)t}\right) 
\end{align}

\noindent Here, $x(t)$ goes extinct in finite time if

\begin{align}\label{ineq:w1}
\sqrt{x(0)} \leq  \dfrac{0.5y(0)}{0.5r+\alpha} 
\end{align}

\noindent Hence
\begin{align}\label{ineq:w2}
K(x(0)) \leq y(0) 
\end{align}

\noindent where $K(x(0))=(r+2\alpha)\sqrt{x(0)}$.

\end{proof}

\begin{example}
Let $\alpha=0.4$ and $\beta=0.65$, then
\begin{equation}
\dfrac{\alpha}{\beta}=0.61538> \dfrac{1}{\sqrt{3}}=0.57735.
\end{equation}
Clearly, the parametric restriction in \eqref{ineq:1} is meet here.
Furthermore, for $r=0.5$ and initial condition $(x,y)=(0.4,1)$
\begin{equation}
K(x(0))=0.82219<1=y(0).
\end{equation}
Obviously, \eqref{ineq:w2} is satisfied. From Figure \ref{fig:separatrix}(c), the trajectory of the initial condition $(0.4,1)$ hits the predator axis in finite time. This is further substantiated with the time series plot in Figure \ref{fig:separatrix}(b).  Hence $P_2$ cannot be globally asymptotically stable. 
\end{example}

\begin{definition}
A separatrix is the boundary dividing two modes of behavior in a phase portrait, see Figure \ref{fig:separatrix}(c).
\end{definition}

\begin{remark}
The inequality in \eqref{ineq:w2} is a sufficient condition needed to disprove the global stability claim by Pal et al. \cite{P17}. The separatrix (stable manifold) from the phase portrait in Figure \ref{fig:separatrix}(c) cannot be computed explicitly hence we use the derived estimate $K(x)=(r+2\alpha)\sqrt{x}$.  Clearly, there is a gap between the separatrix and $K(x)$. Initial conditions chosen above the separatrix goes to prey extinction in finite time. Also, initial conditions chosen below the separatrix goes to the stable interior equilibrium point $P_2$, see Figure \ref{fig:separatrix}(b).

\end{remark}

\begin{figure}[!htb]
\begin{center}
\subfigure[]{
    \includegraphics[width=6.15cm, height=6cm]{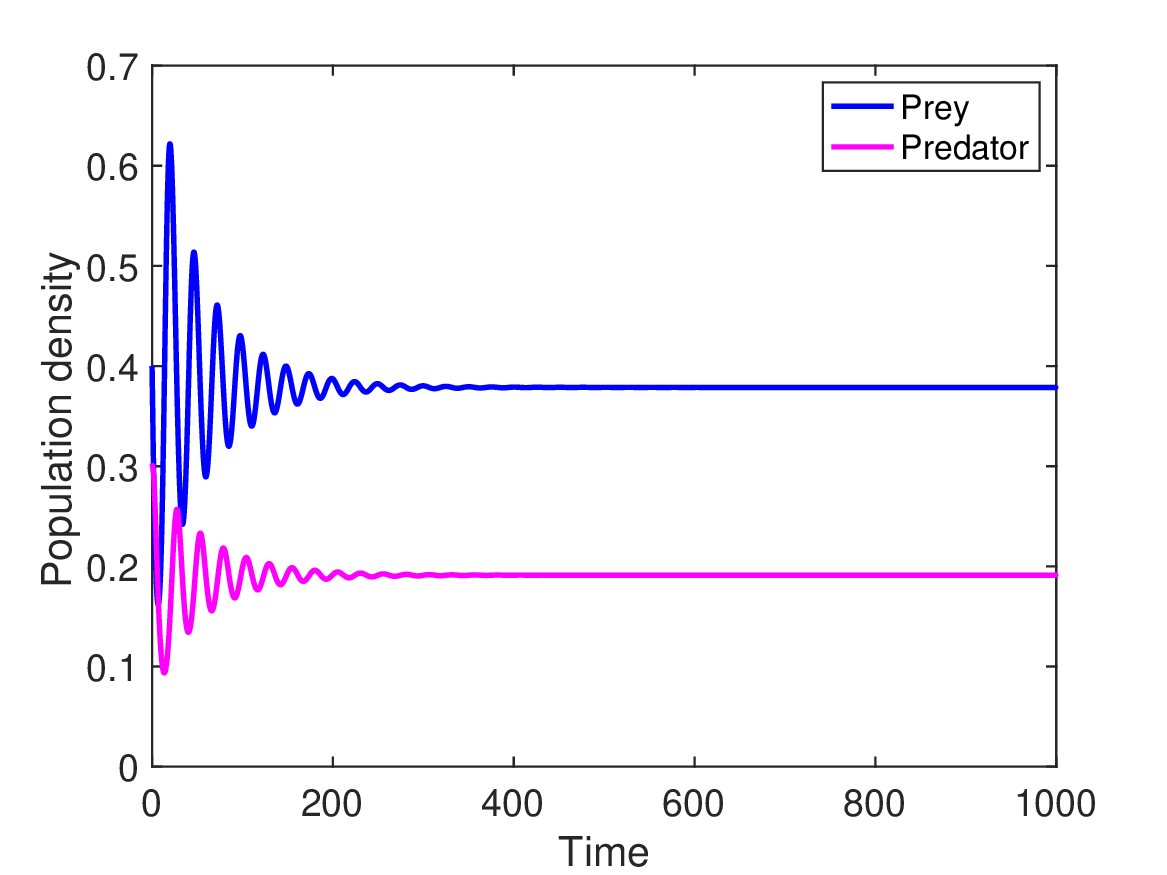}}
\subfigure[]{    
    \includegraphics[width=6.15cm, height=6cm]{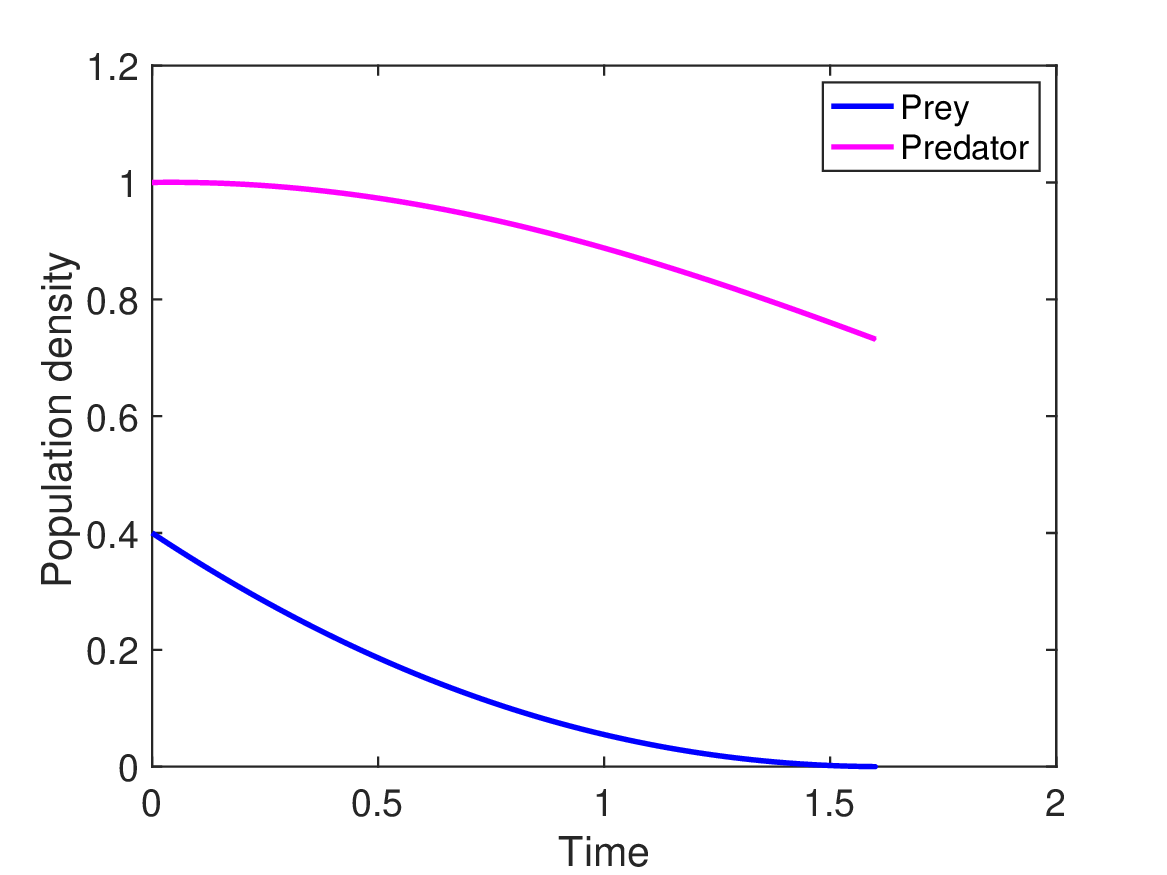}} 
    \subfigure[]{
    \includegraphics[width=6.5cm, height=5cm]{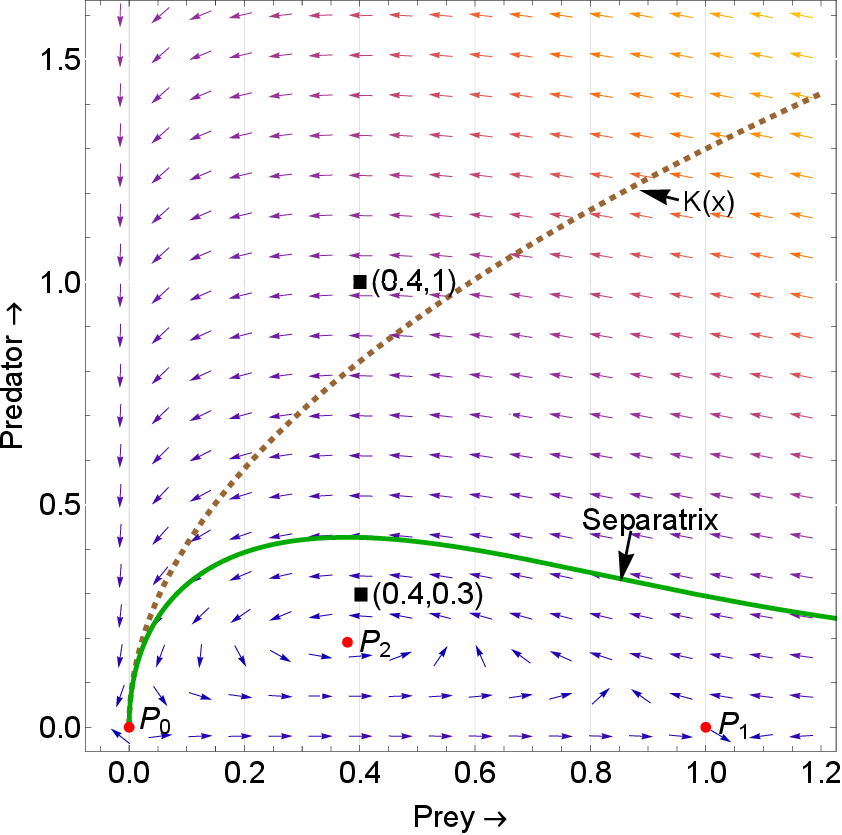}}
 \end{center}
 \caption{(a) Time series plot depicting how the interior equilibrium point $P_2=(0.3787,0.1912)$ stabilizes when the initial condition is chosen as $(0.4,0.3)$  (b) time series plot illustrating finite time extinction at $t=1.624$ of the prey species when the initial condition is chosen as $(0.4,1)$ (c) phase portrait showing the separatrix (green curve), equilibrium points (solid red circles), initial conditions (solid black squares), and the $K(x)$ (brown dashed curve). Here $r=0.5,~\alpha=0.4,$ and $\beta=0.65$. }
      \label{fig:separatrix}
 \end{figure}

\section{Conclusion}
\noindent In this research, we have provided a detailed mathematical proof to show that the interior equilibrium point of the model system \eqref{EquationMain} considered by Pal et al. \cite{P17} can never be globally asymptotically stable. Thus, the global stability does not depend on the parameters $\alpha$ and $\beta$ as claimed by the authors. This key result is corroborated with numerical simulations via Figure \ref{fig:separatrix} and an example. It is important to carefully examine model systems with square root functional response before claiming global existence. The square root functional response is not differentiable at the origin and hence standard stability analysis cannot be carried out. Thus, the solutions are not unique in backward time, see \cite{KRM20}.

\section*{Conflict of interest}
\noindent There is no conflict of interest.

\section*{Author's Contribution}
\noindent KB and KAF contributed equally to this manuscript. All authors read and approved the final manuscript.

\section*{Availability of data and materials}
\noindent Not applicable.



\bibliographystyle{mdpi}

\renewcommand\bibname{References}

\end{document}